\documentclass[12pt]{amsart} 
\usepackage[top=1in, bottom=1in, left=1in, right=1in, footskip=0.5in]{geometry}

\usepackage{times}
\usepackage[table,dvipsnames]{xcolor}

\PassOptionsToPackage{hyphens}{url}\usepackage[colorlinks=true, linkcolor=NavyBlue, citecolor=teal, urlcolor=gray]{hyperref}   
\numberwithin{equation}{section}

\usepackage{tikz-cd}
\usepackage{scrextend}
\usepackage{amssymb,mathrsfs,amsmath}
\usepackage{mathtools}
\usepackage{bbm}
\usepackage{setspace}
\usepackage{enumitem}
\usepackage{hyperref}

\theoremstyle{plain}
\newtheorem{thm}{Theorem}
\newtheorem{prop}{Proposition}[section]
\newtheorem{lem}[prop]{Lemma}

\newtheorem{cor}{Corollary}

\theoremstyle{definition}

\theoremstyle{remark}
\newtheorem*{rem*}{Remark}
\newtheorem*{rems*}{Remarks}

\newtheorem{question}{Question}

\newcommand{\Z}{\mathbb{Z}}

\newcommand{\C}{\mathbb{C}}

\newcommand{\E}{\mathbb{E}}
\renewcommand{\P}{\mathbb{P}}

\newcommand{\CA}{\mathcal{A}}

\newcommand{\CC}{\mathcal{C}}

\newcommand{\CF}{\mathcal{F}}

\let\gcd\relax
\DeclareMathOperator{\gcd}{gcd}

\let\lcm\relax
\DeclareMathOperator{\lcm}{lcm}

\newcommand{\bv}\boldsymbol{}

\newcommand{\eps}\varepsilon

\renewcommand{\leq}{\leqslant}
\renewcommand{\le}{\leqslant}
\renewcommand{\geq}{\geqslant}
\renewcommand{\ge}{\geqslant}

\renewcommand{\mod}[1]{\ \,(\mathrm{mod}\,#1)}

\title{Results for distinct covering systems with a prime minimum modulus $q \ge 11$}
\author{Michael Filaseta}
\author{Alexandros Kalogirou}
\author{Jonah Klein}


\begin{document}

\begin{abstract}
We prove that if the smallest modulus of a covering system is a prime $q \ge 11$, then the sum of the reciprocals of the moduli is at least $1+10^{-33}$. This improves on a previous more general bound of the first two authors in this particular case. We also show that it is impossible to construct a covering system using distinct multiples of $q$ as moduli. Finally, we show as a corollary that it is impossible to construct $10$ distinct covering systems with the property that no modulus is used in more than one of the covering systems. 
\end{abstract}
\maketitle

\section{Introduction}

For $a, d \in \Z$, with $d \ge 1$, the \textit{arithmetic progression} $a \mod d$ is the set
$\{a+kd:k \in \Z\}$. A \textit{covering system} or simply a \textit{covering} is a finite collection of arithmetic progressions with each modulus $d \ge 2$ and with the property that each integer belongs to at least one of the progressions. We say that a covering system is \textit{distinct} if every arithmetic progression in the covering system has a different modulus. 

Covering systems were introduced by Erd\H{o}s \cite{Erdos} in 1950. In the same paper, he asked whether or not for any fixed $M$, it is possible to construct a covering system with distinct moduli, all of which are $>M$. This problem, often referred to as the \textit{minimum modulus problem}, was resolved by Hough \cite{hough} in 2015 in the negative; he showed that any covering system with distinct moduli has a modulus that is smaller than $10^{16}$. 

In 2022, Balister, Bollob\'as, Morris, Sahasrabudhe, and Tiba \cite{Balister} improved on Hough's ideas and reduced this bound to $615999$. The method they used, which they named the \textit{distortion method}, has seen a number of applications in recent years, see for example \cite{BalisterSquareFree, Balister, Cummings, Filaseta, Klein, Li, Li2, Li3, Wang}. For an expository note on the distortion method, see \cite{BalisterExposition}. 

In 1952, Davenport (as reported by Erd\H{o}s \cite{Mirsky}) posed the problem of finding a condition on the minimum modulus $m_0$ of a distinct covering system that would imply that the sum of the reciprocals of the moduli is bounded away from 1. In one of the applications mentioned above, the first two authors \cite{Filaseta} resolved this problem, showing that the sum of the reciprocals of the moduli in a covering system with distinct moduli and minimum modulus at least 5 is bounded away from $1$. This confirmed a conjecture of Erd\H{o}s and Selfridge \cite{Erdos2}. In particular, they proved the following result. 

\begin{thm}[\cite{Filaseta}]\label{FilasetaKalogirou}
For every finite distinct covering system with minimum modulus exceeding 4, the sum of the reciprocals of the moduli is at least 
\[1+\exp(-3.343054 \cdot 10^{21}).\]
\end{thm}

One of our goals in this paper is to show an alternative approach to obtaining a form of this result in the case that the minimum modulus is a prime $\ge 11$. In particular, we prove the following. 

\begin{thm}\label{reciprocalsum}
If $C$ is a covering system with distinct moduli and minimum modulus a prime $q \ge 11$, then the sum of the reciprocals of the moduli in $C$ is at least $1+10^{-33}$. 
\end{thm}

For the proof of Theorem~\ref{reciprocalsum}, we will make direct use of the distortion method. 
Although the first two authors in \cite{Filaseta} also made use of the distortion method, they further exploited some different ideas to achieve their result for general minimum modulus exceeding $4$. 
Our proof here is centered around what is needed to establish Theorem~\ref{reciprocalsum} for $q = 11$ by a direct use of the distortion method.  
The lower bound on the sum of the reciprocals of the moduli  in Theorem~\ref{reciprocalsum} can be improved for primes $q > 11$, but we only focus on showing that a uniform bound of $1+10^{-33}$ can be obtained for all primes $q \ge 11$.
Along the way to proving Theorem \ref{reciprocalsum}, we also prove the following result. 

\begin{thm}\label{no11}
Let $q \ge 11$ be a prime. Then
there is no covering system using only distinct multiples of $q$ as moduli.  
\end{thm}

Instead of directly proving Theorem \ref{reciprocalsum} or Theorem \ref{no11}, we will instead prove the following slightly more general theorem, of which Theorem \ref{reciprocalsum} and Theorem \ref{no11} are both corollaries. 

\begin{thm}\label{main}
Let $q \ge 11$ be a prime. 
Suppose $\CA$ is a finite collection of arithmetic progressions with distinct moduli, minimum modulus $q$, and such that every modulus $m$ that is not divisible by $q$ satisfies $m>M:=10^{33}/q$. Then $\CA$ is not a covering system. 
\end{thm}

It is clear that Theorem \ref{no11} is an immediate consequence of Theorem \ref{main}. The following lemma suffices to see that Theorem \ref{reciprocalsum} also follows from Theorem \ref{main}. This lemma is similar to an observation made in \cite[p.~4950]{Filaseta}.

\begin{lem}\label{bigcoprime}
Suppose $C$ is a covering system with distinct moduli and minimum modulus a prime $q \ge 11$. Suppose further that the sum of the reciprocals of the moduli in $C$ is smaller than $1+10^{-33}$. Then any modulus $m$ of an arithmetic progression in $C$ that is not divisible by $q$ satisfies $m>10^{33}/q$. 
\end{lem}

\begin{proof}
Suppose $m$ is a modulus of a congruence in $C$ with $m$ not divisible by $q$. By the Chinese Remainder Theorem, the arithmetic progression with modulus $m$ intersects with the arithmetic progression with modulus $q$ on a set of density $1/(qm)$. Let $D$ denote the set of moduli appearing in $C$. Because $C$ is a covering system, we deduce that
\[\sum_{d \in D} \frac{1}{d}-\frac{1}{qm} \ge 1,\]
which implies that
\[
\sum_{d \in D} \frac{1}{d} \ge 1+\frac{1}{qm}.
\]
Since $\sum_{d \in D} 1/d < 1 + 1/10^{33}$ by assumption, we deduce that
\[
1+\frac{1}{qm}<1+\frac{1}{10^{33}},
\]
which implies that $m>10^{33}/q$. 
\end{proof}

Another question to which our work here manages to give a partial answer is the following. What is the largest family of distinct covering systems one can construct such that no two of the covering systems have a modulus in common? Recall that we are taking each modulus to be at least $2$ in a covering system, so in particular $x \equiv 0 \pmod{1}$ is not a covering. Let $N$ denote the size of the largest such family. On one hand, a construction of Harrington \cite{Harrington} implies that $N \geq 3$. On the other hand, the bound of $615999$ on the minimum modulus by Balister et al. \cite{Balister} implies the \textit{trivial} upper bound $N \leq 615998$. 
Indeed, if one is constructing a family of distinct covering systems, then each such covering system has a minimum modulus. Since all of these must be distinct, and they must be at least $2$ and at most $615999$, there can be at most $615998$ minimum moduli, and so there are at most $615998$ covering systems in such a family. 

As a corollary of Theorem \ref{no11}, we get the following improvement on an upper bound for $N$.  

\begin{cor}\label{cor}
For any family of 10 covering systems with distinct moduli, there must be at least one integer that is used as a modulus in more than one of the covering systems. In other words, $N \le 9$. 
\end{cor}

\begin{proof}
Assume that there exists a family of $10$ distinct covering systems that share pairwise no common modulus. Let
\[\CF=\{C_1,\ldots,C_{10}\}\]
be such a family. For a covering system 
\[C=\{a_1 \mod {m_1}, \ldots, a_n \mod{m_n}\},\]
define
\[xC+y:=\{xa_1 +y\mod{xm_1},\ldots,xa_n+y \mod{xm_n}\}.\]
Let
\[\CC:=\{0 \mod{11}\}\cup\bigcup_{i=1}^{10}\,(11C_i+i).\]
We claim that $\CC$ is a distinct covering system using only multiples of $11$ as moduli. The fact that $\CC$ is distinct follows from the fact that the covering systems in $\CF$ pairwise share no common modulus. The fact that all moduli in $\CC$ are multiples of $11$ is immediate from the definition of $xC+y$. 

We now show that $\CC$ is indeed a covering system. Let $k \in \Z$. If $11 \mid k$, then $0 \mod{11}$ covers $k$. If $11 \nmid k$, then there is some $i \in \{ 1, 2, \ldots, 10\}$ such that $k \equiv i \mod{11}$. It follows that $(k-i)/11$ is an integer. Because $C_i$ is a covering system, there is some arithmetic progression $a \mod m$ in $C_i$ such that $(k-i)/11 \equiv a \mod m$. It follows that $k \equiv 11a+i \mod{11m}$, i.e. that $k$ is covered by an arithmetic progression in $11C_i+i \in \CC$. Because $k \in \Z$ was arbitrary, we deduce that $\CC$ is indeed a covering system. This contradicts Theorem~\ref{no11} with $q = 11$. 
\end{proof}

\section{The distortion method}\label{sect2}
In this section, we set up some notation and recall some key results of the distortion method. We attempt to stay brief, and refer the reader to one of \cite{BalisterExposition, Balister, Cummings} for a gentler introduction to the distortion method. 

Let $D=\{d_1,\ldots,d_n\}$ be a finite set of integers such that the smallest integer in $D$ is a prime $q \ge 11$ and every integer $d \in D$ that is not divisible by $q$ satistfies $d> M = 10^{33}/q$. Let $\CA=\{A_d:d \in D\}$ be a finite collection of arithmetic progressions with moduli set $D$. Our goal is to show that $\CA$ is not a covering system.

Let $Q$ be the least common multiple of all integers in $D$, and write
\[Q=\prod_{i=1}^J p_i^{\nu_i},\]
where $p_1<p_2<\cdots<p_J$ are the distinct prime divisors of $Q$. 
We may add additional multiples of $q$ to $D$ without affecting what we want to show, so there is no harm in assuming that $p_i$ is the $i^{\rm th}$ prime, and we do so. 
We reveal the arithmetic progressions in stages according to their prime divisors. 
Let $P^{+}(d)$ denote the largest prime divisor of $d$.  
For $j \in \{ 1, 2, \ldots, J\}$, let 
\[Q_j:=\prod_{i=1}^jp_i^{\nu_i}\]
be the $p_j$-smooth part of $Q$,
and
\[B_j:=\bigcup_{\substack{1 \leq i \leq n\\ P^+(d_i)=p_j}}\{a \in \mathbb Z/Q \mathbb Z :a \equiv a_i \mod{d_i}\} \]
be the set of integers modulo $Q$ that are covered by arithmetic progressions that are revealed at stage $j$. 

Following \cite{Balister}, we construct a sequence of probability measures $\P_0,\P_1,\ldots,\P_J$ on $\Z/Q\Z$, defined in terms of some free parameters $\delta_1,\ldots,\delta_J \in [0,1/2]$. We first set up some notation. For every $j \in \{0,1,\ldots,J\}$, let $\pi_j:\Z/Q\Z \to \Z/Q_j\Z$ denote the natural projection. For every $x \in \Z/Q\Z$, we define
\[
F_j(x):=\{y \in \Z/Q\Z:\pi_j(y)=\pi_j(x)\}.
\]
We are now ready to define the measures. We start by letting $\P_0$ be the uniform measure over $\Z/Q\Z$. Let $j \in \{ 1, 2, \ldots, J \}$, and suppose that we have already defined the measure $\P_{j-1}$. 
We define
\[
\alpha_j(x):=\frac{|F_{j-1}(x) \cap B_j|}{|F_{j-1}(x)|}.
\]
For $x \in \Z/Q\Z$, we define $\P_j(x)$ as follows. We condition on whether $\alpha_j(x)<\delta_j$. If $\alpha_j(x)<\delta_j$, we define
\[\P_j(x):=\P_{j-1}(x) \cdot \frac{\mathbbm{1}_{x \notin B_j}}{1-\alpha_j(x)}.\]
If $\alpha_j(x) \geq \delta_j$, we define
\[\P_j(x):=\P_{j-1}(x)\cdot \begin{cases}
\frac{\alpha_j(x)-\delta_j}{\alpha_j(x)(1-\delta_j)} & \text{if } x \in B_j,\\
\frac{1}{1-\delta_j} & \text{if } x \notin B_j.
\end{cases}\]
The fact that these are indeed probability measures is proven in \cite{Balister}. We also introduce the following definition. For any function $f:\Z/Q\Z \to \C$, and any $j \in \{0,1,\ldots,J\}$, we define
\[\E_j[f]:=\sum_{x \in \Z/Q\Z}f(x)\,\P_j(x).\]
Finally, for $k \geq 1$, we define the $k$-\textit{th moment}
\[M_j^{(k)}:=\E_{j-1}[\alpha_j^k].\]
We now recall some useful facts based on \cite{Balister}.  The next result is Theorem~3.1 in \cite{Balister}.

\begin{lem}[\cite{Balister}] \label{maindistortion}
If 
\[
\sum_{\substack{1 \le i \le J\\ \delta_{i} = 0}} M_i^{(1)} + \sum_{\substack{1 \le i \le J\\ \delta_{i} \ne 0}} \min\bigg(M_i^{(1)},\frac{M_i^{(2)}}{4\delta_i(1-\delta_i)}\bigg) < 1,
\]
then $\CA$ does not cover the integers. 
\end{lem}

We will take $\delta_{i} = 0$ for $1 \le i \le 18$.  Then it follows from the definition of $\mathbb P_i$ that $\mathbb P_i$ is the uniform measure over $\mathbb Z/Q\mathbb Z$ for every $i \in \{1, 2, \ldots, 18\}$.  
Thus, for such an $i$, the probability an $x \in \Z/Q\Z$ falls into some fixed residue class modulo $d \in D$ is $1/d$.
As a consequence of Theorem~3.2 (and its proof) in \cite{Balister}, we obtain the following. 

\begin{lem}[\cite{Balister}]\label{mainmomentbounds}
Given the above notation, if $\delta_{i} = 0$ for $1 \le i \le 18$, then
\begin{equation}\label{mainmomentbounds:eq1}
\sum_{i=1}^{18}M_i^{(1)} \leq \sum_{\substack{d\in D\\d \text{ }p_{18}\text{-smooth}}} \frac{1}{d},
\end{equation}
and
\begin{align*}
M_i^{(2)} 
&\leq \frac{1}{(p_i-1)^2}\sum_{\substack{m_1,m_2|Q_{i-1}\\q \,\mid\,\gcd(m_1,m_2)}}\frac{\nu(\lcm(m_1,m_2))}{\lcm(m_1,m_2)} + 
\sum_{j_{1}=1}^{\infty} \sum_{j_{2}=1}^{\infty}
\dfrac{1}{p_i^{j_1+j_2}}
\sum_{\substack{m_1p_i^{j_1},m_2p_i^{j_2} \in D\\m_1,m_2|Q_{i-1}\\q \,\nmid\, \gcd(m_1,m_2)}}\frac{\nu(\lcm(m_1,m_2)}{\lcm(m_1,m_2)} \\[5pt]
&\leq\frac{1}{(p_i-1)^2}\sum_{m_1,m_2|Q_{i-1}}\frac{\nu(\lcm(m_1,m_2))}{\lcm(m_1,m_2)}
\leq \frac{1}{(p_i-1)^2}\prod_{j <i}\bigg(1+\frac{3p_j-1}{(1-\delta_j)(p_j-1)^2}\bigg), 
\end{align*}
where $\nu$ is a multiplicative function defined at prime powers by
\[
\nu(p_i^m)=\frac{1}{1-\delta_i}.
\]
\end{lem}

Our main goal for what follows is to bound the first and second moments, taking advantage of the specific conditions we have on $D$. We must also carefully choose the $\delta_i$'s in order to get the argument to work. After choosing $18$ with some experimentation in its role above, we optimized our choices of $\delta_i$ numerically to $3$ digits for the ranges shown below as this both was sufficient for what we wanted to show and makes it easy to verify our results.
The $\delta_{i}$'s we chose are
\begin{gather*}
\delta_{i}=0  \quad\text{ for $1 \le i \le 18$},  \\[4pt]
\delta_{19}=0.263, \quad
\delta_{20}=0.266, \quad
\delta_{21}=0.271, \quad
\delta_{22}=0.271, \quad
\delta_{23}=0.273,  \\[4pt]
\delta_{24}=0.273, \quad
\delta_{25}=0.272, \quad
\delta_{26}=0.274, \quad
\delta_{27}=0.278, \quad
\delta_{28}=0.280,  \\[4pt]
\delta_{29}=0.283, \quad
\delta_{30}=0.286, \quad
\delta_{31}=0.280, \quad
\delta_{32}=0.282, \quad
\delta_{33}=0.282,  \\[4pt]
\delta_{34}=0.285, \quad
\delta_{35}=0.283, \quad
\delta_{36}=0.286, \quad
\delta_{37}=0.286, \quad
\delta_{38}=0.287, \quad
\delta_{39}=0.288, \\[4pt]
\delta_{i}=0.292 \quad \text{for $40 \le i \le 50$}, \qquad
\delta_{i}=0.305 \quad \text{for $51 \le i \le 100$},  \\[4pt]
\delta_{i}=0.322, \quad \text{for $101 \le i \le 200$}, \qquad
\delta_{i}=0.338 \quad \text{for $201 \le i \le 500$},  \\[4pt]
\delta_{i}=0.353 \quad \text{for $501 \le i \le 1000$}, \qquad
\delta_{i}=0.365 \quad \text{for $1001 \le i \le 2000$},  \\[4pt]
\delta_{i}=0.381 \quad \text{for $2001 \le i \le 10000$}, \qquad
\delta_{i}=0.404 \quad \text{for $10001 \le i \le 20000$},  \\[4pt]
\delta_{i}=0.426 \quad \text{for $20001 \le i \le 50000$}, \qquad
\delta_{i}=0.467 \quad \text{for $50001 \le i \le 100000$}, \\[4pt]
\delta_{i} = 0.5 \quad \text{for all $i > 10^{5}$.}
\end{gather*}

\section{Bounding the Moments}
\subsection{First Moment Bounds}
We fix a prime $q \ge 11$, and let $\mathcal A$ be as in Theorem~\ref{main}.  
As discussed in the introduction, we know from \cite{Balister} that $\mathcal A$ is not a covering if the minimum modulus $q$ exceeds $615999$.  Thus, Theorem~\ref{main} follows in this case, so we consider now $q \le 615997$, the largest prime less than or equal to $615999$.
To bound the first moments, it will be useful to partition the elements of $D$ into two sets, depending on whether or not they are divisible by $q$. With this in mind, let 
\[D_1=\{d \in D: q \mid d\},\]
and
\[D_2=\{d\in D: q \nmid d\}.\]
We are wanting to prove Theorem~\ref{main}, so we are assuming that
each $d \in D_2$ satisfies $d> M =10^{33}/q$. 
We make this partition, as we wish to split up the bound on the right-hand side of \eqref{mainmomentbounds:eq1} into two sums, conditioning on whether $d \in D_1$ or $d \in D_2$. We deal with the first of these cases as follows. 

\begin{lem}\label{moment1d1}
Given the above notation, we have
\[
\sum_{\substack{d \in D_1\\d \text{ } p_{18}-\text{smooth}}} \frac{1}{d} \leq 
\begin{cases}
0.690864963 &\text{if } q = 11, \\
0.584578046 &\text{if } 13 \le q \le 61, \\
0 &\text{if } q > 61.
\end{cases}
\]
\end{lem}

\begin{proof}
Since $p_{18} = 61$ and $q$ divides every $d \in D_{1}$, we see the sum is empty if $q > 61$, giving the last bound above.  
For $11 \le q \le 61$, we make use of the inequality 
\[
\sum_{\substack{d \in D_1\\d \text{ } p_{18}-\text{smooth}}} \frac{1}{d} 
\leq \bigg(\frac{1}{q}+\frac{1}{q^2}+\cdots\bigg)\prod_{\substack{p \leq p_{18}\\p \neq q}}\bigg(1+\frac{1}{p}+\frac{1}{p^2}+\cdots\bigg)=\frac{1}{q-1}\prod_{\substack{p \leq p_{18}\\p \neq q}}\frac{p}{p-1}.
\]
A direct computation yields the bounds indicated in the case $q = 11$ and $13 \le q \le 61$. 
\end{proof}

We now look at the contribution of the sum over $d \in D_2$.

\begin{lem}\label{moment1d2}
Given the above notation, we have
\[
\sum_{\substack{d \in D_2\\d \text{ } p_{18}-\text{smooth}}}\frac{1}{d} < 7.207 \cdot 10^{-11}.
\]
\end{lem}

\begin{proof}
By assumption, any $d \in D_2$ satisfies $d>M$.  We deduce, for any $\epsilon>0$ and any odd prime $q \ge 11$, that
\[
\sum_{\substack{d \in D_2\\d \text{ } p_{18}-\text{smooth}}}\frac{1}{d}
< \sum_{\substack{d \in D_2\\d \text{ } p_{18}-\text{smooth}}}\bigg(\frac{d}{M}\bigg)^{\epsilon}\frac{1}{d}
< \frac{1}{M^{\epsilon}}\prod_{1 \leq i \leq 18}\frac{p_i^{1-\epsilon}}{p_i^{1-\epsilon}-1}. 
\]
Using $M = 10^{33}/q \ge 10^{27}/615997$ and taking $\epsilon=3/4$, a direct computation gives the claimed upper bound. 
\end{proof}

Together, Lemma~\ref{moment1d1},  Lemma~\ref{moment1d2} and Lemma~\ref{mainmomentbounds}, 
with $\delta_{i} = 0$ for $1 \le i \le 18$, imply 
\begin{equation}\label{firstmomentbound}
\sum_{i=1}^{18}M_i^{(1)} < 
\begin{cases}
0.690864964 &\text{if } q = 11, \\
0.584578047 &\text{if } 13 \le q \le 61, \\
7.207 \cdot 10^{-11} &\text{if } q > 61.
\end{cases}
\end{equation}

\subsection{Second Moment Bounds}
We now focus our attention on bounding the second moments. Towards this goal, let
\[T_{i,1}:=\frac{1}{(p_i-1)^2}\sum_{\substack{m_1,m_2|Q_{i-1}\\q \mid \gcd(m_1,m_2)}}\frac{\nu(\lcm(m_1,m_2))}{\lcm(m_1,m_2)},\]
and
\[
T_{i,2}:=
\sum_{j_{1}=1}^{\infty} \sum_{j_{2}=1}^{\infty}
\dfrac{1}{p_i^{j_1+j_2}}
\sum_{\substack{m_1p_i^{j_1},m_2p_i^{j_2} \in D\\m_1,m_2|Q_{i-1}\\q \,\nmid\, \gcd(m_1,m_2)}}\frac{\nu(\lcm(m_1,m_2)}{\lcm(m_1,m_2)}.
\]
Recalling Lemma~\ref{mainmomentbounds}, we find that
\[M_i^{(2)} \leq T_{i,1}+T_{i,2}.\]

We turn to providing bounds for $T_{i,1}$ and $T_{i,2}$. To do so, we follow in part ideas from the proof of Lemma 3.7 in \cite{Balister}. 

\begin{lem}\label{T1bound}
Let $\pi(q) = i_{0}$, and let $i \ge 19$.  With the above notation, the inequality
\[
T_{i,1} \leq 
\begin{cases}
0 &\text{if } i \le i_0 \\
\displaystyle \frac{q+1}{(q-1)^{2}(1-\delta_{i_{0}})(p_i-1)^2}\prod_{\substack{1 \leq j <i\\j \neq i_{0}}}\bigg(1+\frac{3p_j-1}{(1-\delta_j)(p_j-1)^2}\bigg) &\text{if } i > i_0
\end{cases}
\]
holds.
\end{lem}

\begin{proof}
Observe that $\pi(q) = i_{0}$ implies that $\delta_{i_{0}}$ is associated with the $i_{0}^{\rm th}$ prime $q$.
The definition of $Q_{i-1}$ implies that, in order for the sum in the definition of $T_{i,1}$ to be non-zero, we must have $i_0 \le i-1$ or, equivalently, $i > i_{0}$.  
Let 
\[
\chi(m):=|\{(m_1,m_2) \in (\mathbb Z^{+})^{2}: q \mid \gcd(m_1,m_2), \lcm(m_1,m_2)=m\}|,
\]
and note that
\[
\sum_{\substack{m_1,m_2 \mid Q_{i-1}\\q \mid \gcd(m_1,m_2)}}\frac{\nu(\lcm(m_1,m_2))}{\lcm(m_1,m_2)}=\sum_{\substack{m \mid Q_{i-1}\\q \mid m}}\frac{\chi(m)\nu(m)}{m}.
\]
As noted in \cite{Balister}, without the condition $q \mid \gcd(m_1,m_2)$ in the definition of $\chi(m)$, the function $\chi$ is multiplicative and satisfies $\chi(p^t)=2t+1$ for primes $p$ and positive integers $t$.  In our case, the function $\chi(m)$ is still multiplicative, $\chi(p^t)=2t+1$ for $p \ne q$, and $\chi(q^t)=2t-1$. 
It follows that $\chi(m)\nu(m)/m$ is multiplicative as well, and that 
\[
\sum_{\substack{m \mid Q_{i-1}\\q \mid m}}\frac{\chi(m)\nu(m)}{m}
\le \Bigg(\sum_{t=1}^\infty \frac{2t-1}{(1-\delta_{i_{0}})q^t}\Bigg)\prod_{\substack{j<i\\j \neq i_{0}}}\Bigg(1+\sum_{t=1}^{\infty}\frac{2t+1}{(1-\delta_j)p_j^t}\Bigg).
\]
As noted in \cite{Balister}, one can show that
\[\sum_{t=1}^{\infty}\frac{2t+1}{p^t}=\frac{3p-1}{(p-1)^2}. \]
This implies that
\[
\sum_{t=1}^{\infty}\frac{2t-1}{q^t}=\frac{3q-1}{(q-1)^2}-\frac{2}{q-1}=\frac{q+1}{(q-1)^{2}}. 
\]
Hence, we obtain
\[
\sum_{\substack{m \mid Q_{i-1} \\q \mid m}}\frac{\chi(m)\nu(m)}{m} 
\leq \frac{q+1}{(1-\delta_{i_{0}})(q-1)^{2}}\prod_{\substack{1 \leq j <i\\j \neq i_{0}}}\bigg(1+\frac{3p_j-1}{(1-\delta_j)(p_j-1)^2}\bigg),
\]
which implies what we wanted to show. 
\end{proof}

Observe that $q \ge 11$ implies $i_0 \ge 5$.
Also, the bound in Lemma~\ref{T1bound} on $T_{i,1}$ for $i > i_0$ can be used for a bound when $i \le i_0$. 
The dependence on $q$ in Lemma~\ref{T1bound} is clearer if we write this bound on $T_{i,1}$ as
\begin{gather*}
\frac{q+1}{(1-\delta_{i_{0}})(q-1)^{2}} \bigg(1+\frac{3q-1}{(1-\delta_{i_{0}})(q-1)^2}\bigg)^{-1} \cdot \dfrac{1}{(p_i-1)^2} \prod_{1 \leq j < i} \bigg(1+\frac{3p_j-1}{(1-\delta_j)(p_j-1)^2} \bigg) \\[5pt]
=\frac{q+1}{(1-\delta_{i_{0}})(q-1)^{2}+3q-1} \cdot \dfrac{1}{(p_i-1)^2} \prod_{1 \leq j < i} \bigg(1+\frac{3p_j-1}{(1-\delta_j)(p_j-1)^2} \bigg).
\end{gather*}
Defining $f(q) = (q+1)/((1-\delta_{i_{0}})(q-1)^{2}+3q-1)$, where $q \ge 11$, one can check that $f(q) = 1/q$ if $\delta_{i_{0}} = 0$ and $f(q) \le 2/(q+2)$ if $\delta_{i_{0}} \in [0,1/2]$.  Since $\delta_{5} = \delta_{6} = \cdots = \delta_{18} = 0$ and all other choices of $\delta_{i}$ are in $[0,1/2]$, we see that $f(11) = 1/11$ and $f(q) \le 1/13$ for all $q \ge 13$.  
A direct computation using the choices for $\delta_{i}$ at the end of Section~\ref{sect2} gives us 
\begin{equation}\label{t1bound}
\sum_{i=19}^{25997}\frac{T_{i,1}}{4\delta_i(1-\delta_i)} <
\begin{cases}
0.227843431 &\text{if } q = 11 \\
0.192790596 &\text{if } q \ge 13.
\end{cases}
\end{equation}
We now look at bounding $T_{i,2}$. We first consider the case where $p_i=q$.

\begin{lem}\label{pi=qlem}
Given the conditions of Theorem~\ref{main}, if $p_i=q$ where $19 \le i \le 25997$, then
\[
\frac{T_{i,2}}{4\delta_i(1-\delta_i)} \leq 0.079236605.
\]
\end{lem}

\begin{proof}
By Lemma \ref{mainmomentbounds}, we have that
\[\frac{T_{i,2}}{4\delta_i(1-\delta_i)} \leq \frac{1}{4\delta_i(1-\delta_i)}\frac{1}{(p_i-1)^2}\prod_{j <i}\bigg(1+\frac{3p_j-1}{(1-\delta_j)(p_j-1)^2}\bigg). \]
Computing this directly for all $19 \leq i \leq 25997$, we find that for any such $i$, the right-hand side above is always bounded by $0.079236605$. 
\end{proof}

\begin{lem}\label{T2bound}
Set $M_{0} = (3.3 \times 10^{27})/q$.  
Given the conditions of Theorem~\ref{main} and $19 \le i \le 101$, with $p_i \neq q$, the inequality
\[
T_{i,2} \leq \frac{1}{M_{0}^{\epsilon}(p_{i}-1)^{2}}\prod_{j <i}\bigg(1+\frac{3p_j^{1-\epsilon}-1}{(1-\delta_j)(p_j^{1-\epsilon}-1)^2}\bigg)
+ \dfrac{2p_{i}^{2}-1}{p_{i}^{4} (p_{i}-1)^{2}} \prod_{j <i}\bigg(1+\frac{3p_j-1}{(1-\delta_j)(p_j-1)^2}\bigg)
\]
holds for every $\epsilon \in (0,1)$. Similarly, if $102 \le i \le 25997$, with $p_i \neq q$, the inequality
\[
T_{i,2} \leq \frac{1}{M_{0}^{\epsilon}(p_{i}-1)^{2}}\prod_{j <i}\bigg(1+\frac{3p_j^{1-\epsilon}-1}{(1-\delta_j)(p_j^{1-\epsilon}-1)^2}\bigg)
+\dfrac{2p_{i}-1}{p_{i}^{2} (p_{i}-1)^{2}} \prod_{j <i}\bigg(1+\frac{3p_j-1}{(1-\delta_j)(p_j-1)^2}\bigg)
\]
holds for every $\epsilon \in (0,1)$. 
\end{lem}

\begin{proof}
Recall that
\[
T_{i,2} =
\sum_{j_{1}=1}^{\infty} \sum_{j_{2}=1}^{\infty}
\dfrac{1}{p_i^{j_1+j_2}}
\sum_{\substack{m_1p_i^{j_1},m_2p_i^{j_2} \in D\\m_1,m_2|Q_{i-1}\\q \,\nmid\, \gcd(m_1,m_2)}}\frac{\nu(\lcm(m_1,m_2)}{\lcm(m_1,m_2)}.
\]
Because $p_i \neq q$ and $q \nmid \gcd(m_1,m_2)$, we necessarily have that $p_i^{j_k}m_k > M$ for some $k \in \{1,2\}$.   
We consider two cases depending on $i$. 

\smallskip \noindent
Case~1. 
$19 \le i \le 101$.

\smallskip
Either 
\[
\max\{j_{1}, j_{2}\} \ge 3
\]
or
\[
\max\{m_1,m_2\} > \dfrac{M}{p_{i}^2} \ge \dfrac{M}{p_{101}^2} = \dfrac{M}{547^2} > \dfrac{3.3 \times 10^{27}}{q} = M_{0}.
\]
We account for $\max\{j_{1}, j_{2}\} \ge 3$ by observing that
\[
\mathop{\sum_{j_{1}=1}^{\infty} \sum_{j_{2}=1}^{\infty}}_{\max\{j_1,j_2\} \ge 3}
\dfrac{1}{p_i^{j_1+j_2}} 
= \sum_{j_{1}=3}^{\infty} \sum_{j_{2}=1}^{\infty}
\dfrac{1}{p_i^{j_1+j_2}}
+ \sum_{j_{1}=1}^{\infty} \sum_{j_{2}=3}^{\infty}
\dfrac{1}{p_i^{j_1+j_2}}
- \sum_{j_{3}=3}^{\infty} \sum_{j_{2}=3}^{\infty}
\dfrac{1}{p_i^{j_1+j_2}}
= \dfrac{2p_{i}^{2}-1}{p_{i}^{4} (p_{i}-1)^{2}}
\]
and applying Lemma~\ref{mainmomentbounds} to obtain
\begin{align*}
\mathop{\sum_{j_{1}=1}^{\infty} \sum_{j_{2}=1}^{\infty}}_{\max\{j_1,j_2\} \ge 3}
\dfrac{1}{p_i^{j_1+j_2}}
\sum_{\substack{m_1p_i^{j_1},m_2p_i^{j_2} \in D\\m_1,m_2|Q_{i-1}\\q \,\nmid\, \gcd(m_1,m_2)}}\frac{\nu(\lcm(m_1,m_2)}{\lcm(m_1,m_2)}
&\le \mathop{\sum_{j_{1}=1}^{\infty} \sum_{j_{2}=1}^{\infty}}_{\max\{j_1,j_2\} \ge 3}
\dfrac{1}{p_i^{j_1+j_2}}
\sum_{m_1,m_2|Q_{i-1}}\frac{\nu(\lcm(m_1,m_2)}{\lcm(m_1,m_2)} \\[5pt]
&\le \dfrac{2p_{i}^{2}-1}{p_{i}^{4} (p_{i}-1)^{2}} \prod_{j <i}\bigg(1+\frac{3p_j-1}{(1-\delta_j)(p_j-1)^2}\bigg),
\end{align*}
which gives us the second term in the bound we wanted to establish.

For $\max\{m_1,m_2\} > M_{0}$, we use that
\[
\sum_{j_{1}=1}^{\infty} \sum_{j_{2}=1}^{\infty}
\dfrac{1}{p_i^{j_1+j_2}} = \dfrac{1}{(p_i - 1)^2}
\]
and
\[
\sum_{\substack{m_1p_i^{j_1},m_2p_i^{j_2} \in D\\m_1,m_2|Q_{i-1} \\ \max\{m_1,m_2\} > M_{0}}} \frac{\nu(\lcm(m_1,m_2))}{\lcm(m_1,m_2)}
\le \sum_{\substack{m_1,m_2|Q_{i-1} \\ \lcm(m_1,m_2) > M_{0}}} \frac{\nu(\lcm(m_1,m_2))}{\lcm(m_1,m_2)}
= \sum_{\substack{m \mid Q_{i-1}\\m >M_{0}}}\frac{\chi(m)\nu(m)}{m},
\]
where now
\[
\chi(m):=|\{(m_1,m_2) \in (\mathbb Z^{+})^{2}: \lcm(m_1,m_2)=m\}|
\]
is the multiplicative function used in \cite{Balister} that satisfies $\chi(p^t)=2t+1$ for all primes $p$ and positive integers $t$.
For any $\epsilon>0$, we deduce that
\[
\sum_{\substack{m \mid Q_{i-1}\\m >M}_{0}}\frac{\chi(m)\nu(m)}{m}\leq\sum_{m \mid Q_{i-1}}\frac{\chi(m)\nu(m)}{m}\bigg(\frac{m}{M_{0}}\bigg)^{\epsilon}=\frac{1}{M_{0}^{\epsilon}}\sum_{m \mid Q_{i-1}}\frac{\chi(m)\nu(m)}{m^{1-\epsilon}}. 
\]
Now, the function $\chi(m)\nu(m)/m^{1-\epsilon}$ is multiplicative. A similar analysis as in the proof of Lemma~\ref{T1bound} gives us that
\begin{align*}
\frac{1}{M_{0}^{\epsilon}}\sum_{m \mid Q_{i-1}}\frac{\chi(m)\nu(m)}{m^{1-\epsilon}} &\leq \frac{1}{M_{0}^{\epsilon}}\prod_{j <i}\Bigg(1+\sum_{t \geq 1} \frac{2t+1}{(1-\delta_j)p_j^{t(1-\epsilon)}}\Bigg) \\[5pt]
&=\frac{1}{M_{0}^{\epsilon}}\prod_{j <i}\bigg(1+\frac{3p_j^{1-\epsilon}-1}{(1-\delta_j)(p_j^{1-\epsilon}-1)^2}\bigg),
\end{align*}
which leads to the first term in the bound we wanted to establish and finishes the proof for Case~1. 

\smallskip \noindent
Case~2. 
$102 \le i \le 25997$.

\smallskip
Here, we use that either 
\[
\max\{j_{1}, j_{2}\} \ge 2
\]
or
\[
\max\{m_1,m_2\} > \dfrac{M}{p_{i}} \ge \dfrac{M}{p_{25997}} = \dfrac{M}{299993} > \dfrac{3.3 \times 10^{27}}{q} = M_{0}.
\]
For $\max\{j_{1}, j_{2}\} \ge 2$ in this case, we use
\[
\mathop{\sum_{j_{1}=1}^{\infty} \sum_{j_{2}=1}^{\infty}}_{\max\{j_1,j_2\} \ge 2}
\dfrac{1}{p_i^{j_1+j_2}} 
= \sum_{j_{1}=2}^{\infty} \sum_{j_{2}=1}^{\infty}
\dfrac{1}{p_i^{j_1+j_2}}
+ \sum_{j_{1}=1}^{\infty} \sum_{j_{2}=2}^{\infty}
\dfrac{1}{p_i^{j_1+j_2}}
- \sum_{j_{3}=2}^{\infty} \sum_{j_{2}=2}^{\infty}
\dfrac{1}{p_i^{j_1+j_2}}
= \dfrac{2p_{i}-1}{p_{i}^{2} (p_{i}-1)^{2}}
\]
and apply Lemma~\ref{mainmomentbounds} to obtain
\begin{align*}
\mathop{\sum_{j_{1}=1}^{\infty} \sum_{j_{2}=1}^{\infty}}_{\max\{j_1,j_2\} \ge 2}
\dfrac{1}{p_i^{j_1+j_2}}
\sum_{\substack{m_1p_i^{j_1},m_2p_i^{j_2} \in D\\m_1,m_2|Q_{i-1}\\q \,\nmid\, \gcd(m_1,m_2)}}\frac{\nu(\lcm(m_1,m_2)}{\lcm(m_1,m_2)}
&\le \mathop{\sum_{j_{1}=1}^{\infty} \sum_{j_{2}=1}^{\infty}}_{\max\{j_1,j_2\} \ge 2}
\dfrac{1}{p_i^{j_1+j_2}}
\sum_{m_1,m_2|Q_{i-1}}\frac{\nu(\lcm(m_1,m_2)}{\lcm(m_1,m_2)} \\[5pt]
&\le \dfrac{2p_{i}-1}{p_{i}^{2} (p_{i}-1)^{2}} \prod_{j <i}\bigg(1+\frac{3p_j-1}{(1-\delta_j)(p_j-1)^2}\bigg).
\end{align*}
For $102 \le i \le 25997$ and $\max\{m_1,m_2\} \ge M_{0}$, the argument is the same as in the previous case, completing the argument for $102 \le i \le 25997$.
\end{proof}

To take advantage of Lemma~\ref{T2bound}, we use a direct calculation to obtain  
\begin{equation}\label{t2secterms}
\begin{split}
\sum_{i=19}^{101} &\,\dfrac{2p_{i}^{2}-1}{4 \delta_{i} (1-\delta_{i}) p_{i}^{4} (p_{i}-1)^{2}} \prod_{j <i}\bigg(1+\frac{3p_j-1}{(1-\delta_j)(p_j-1)^2}\bigg) \\[5pt]
&\qquad+ \sum_{i=102}^{25997} \dfrac{2p_{i}-1}{4 \delta_{i} (1-\delta_{i}) p_{i}^{2} (p_{i}-1)^{2}} \prod_{j <i}\bigg(1+\frac{3p_j-1}{(1-\delta_j)(p_j-1)^2}\bigg)
\le 0.001358038,
\end{split}
\end{equation}
which holds independent of $q$.  Next we bound
\[
\sum_{i = 19}^{25997}
 \dfrac{1}{4 \delta_{i} (1-\delta_{i}) M_{0}^{\epsilon_{i}} (p_{i}-1)^{2}}\prod_{j <i}\bigg(1+\frac{3p_j^{1-\epsilon_{i}}-1}{(1-\delta_j)(p_j^{1-\epsilon_{i}}-1)^2}\bigg),
\]
where we can choose $\epsilon_{i} \in (0,1)$ as we want.  We chose
\begin{gather*}
\epsilon_{i} = 0.222 \quad \text{for $19 \le i \le 281$}, \qquad
\epsilon_{i} = 0.159 \quad \text{for $282 \le i \le 730$}, \\[4pt]
\epsilon_{i} = 0.117 \quad \text{for $731 \le i \le 1700$}, \qquad
\epsilon_{i} = 0.087 \quad \text{for $1701 \le i \le 3800$}, \\[4pt]
\epsilon_{i} = 0.062 \quad \text{for $3801 \le i \le 10000$}, \qquad
\epsilon_{i} = 0.041 \quad \text{for $10001 \le i \le 25997$}.
\end{gather*}
We justify that
\begin{equation}\label{t2initbound}
\sum_{i = 19}^{25997}
 \dfrac{\prod_{j <i}\bigg(1+\frac{3p_j^{1-\epsilon_{i}}-1}{(1-\delta_j)(p_j^{1-\epsilon_{i}}-1)^2}\bigg)}{4 \delta_{i} (1-\delta_{i}) M_{0}^{\epsilon_{i}} (p_{i}-1)^{2}} < 
\begin{cases}
0.062563774 &\text{if } q = 11, \\
0.091511777 &\text{if } 13 \le q \le 61, \\
0.708631105 &\text{if } q > 61.
\end{cases}
\end{equation}
The first computation with $q=11$ comes from a direct computation using the values of $\epsilon_{i}$ indicated above.  
The dependence on $q$ in the sum is solely in the value of $M_{0} = (3.3 \times 10^{27})/q$.  
Having obtained a savings already in our previous estimates involving $q \ge 13$, we can be rather remiss in our estimates here.  As the largest value of $\epsilon_i$ chosen above is $0.222$ and the bound $0.062563774$ on the sum above was obtained when $M_{0} = (3.3 \times 10^{27})/11$, we deduce for general $q \ge 11$ that
\begin{align*}
&\sum_{i=19}^{25997} \frac{1}{4\delta_i(1-\delta_i) M_{0}^{\epsilon_i} (p_i-1)^2} \prod_{j <i}\bigg(1+\frac{3p_j^{1-\epsilon_{i}}-1}{(1-\delta_j)(p_j^{1-\epsilon_{i}}-1)^2}\bigg) \\[5pt]
&\quad\le \sum_{i=19}^{25997} \dfrac{(3.3 \cdot 10^{27}/11)^{\epsilon_i}}{(3.3 \cdot 10^{27}/q)^{\epsilon_i} } \frac{1}{4\delta_i(1-\delta_i) (3.3 \cdot 10^{27}/11)^{\epsilon_i} (p_i-1)^2} \prod_{j <i}\bigg(1+\frac{3p_j^{1-\epsilon_{i}}-1}{(1-\delta_j)(p_j^{1-\epsilon_{i}}-1)^2}\bigg) \\[5pt]
&\quad\le \bigg( \dfrac{q}{11} \bigg)^{0.222} \,\sum_{i=19}^{25997} \frac{1}{4\delta_i(1-\delta_i)(3.3 \cdot 10^{27}/11)^{\epsilon_i}(p_i-1)^2} \prod_{j <i}\bigg(1+\frac{3p_j^{1-\epsilon_{i}}-1}{(1-\delta_j) (p_j^{1-\epsilon_{i}}-1)^2}\bigg).
\end{align*}
Hence, for $q \ge 11$, we see that
\[
\sum_{i=19}^{25997} \frac{1}{4\delta_i(1-\delta_i)M_{0}^{\epsilon_i} (p_i-1)^2} \prod_{j <i}\bigg(1+\frac{3p_j^{1-\epsilon_{i}}-1}{(1-\delta_j)(p_j^{1-\epsilon_{i}}-1)^2}\bigg) \le \bigg( \dfrac{q}{11} \bigg)^{0.222} (0.062563774).
\]
Recalling that $q \le 615997$, a direct computation replacing $q$ above with $q = 61$ and $q = 615997$ gives \eqref{t2initbound}.  Combining \eqref{t2secterms},  \eqref{t2initbound} and Lemma~\ref{pi=qlem} (in the case $q >61$), we deduce that
\begin{equation}\label{t2bound}
\sum_{i=19}^{25997}\frac{T_{i,2}}{4\delta_i(1-\delta_i)} < 
\begin{cases}
0.063921812 &\text{if } q = 11, \\
0.092869815 &\text{if } 13 \le q \le 61, \\
0.789225748 &\text{if } q > 61.
\end{cases}
\end{equation}

We turn now to the tail end estimates where $i > 10^{5}$.  This part will be independent of $q$.
We begin here with the first portion of the tail.  

\begin{lem}\label{firsttailbound}
With the notation above, we have
\[
\sum_{i=25998}^{2.5\cdot 10^8} \frac{M_i^{(2)}}{4\delta_i(1-\delta_i)}< 0.014179594.
\]
\end{lem}

\begin{proof}
Observe that $\delta_i=1/2$ for $i > 10^5$, implying $4\delta_i(1-\delta_i)=1$ for each such expression in the sum.
We make use of the bound
\begin{equation}\label{lem2bd}
M_i^{(2)} \leq \frac{1}{(p_i-1)^2}\prod_{j <i}\bigg(1+\frac{3p_j-1}{(1-\delta_j)(p_j-1)^2}\bigg)
\end{equation}
from Lemma~\ref{mainmomentbounds}.
We compute the sum in the lemma by first computing the product above for $j \le 25996$.  Then with each value of $i$ in the sum in order, we need only take the product already computed times the factor 
\[
1+\frac{3p_{i-1}-1}{(1-\delta_{i-1})(p_{i-1}-1)^2},
\]
to get the desired product for $i$.
Summing this product divided by $4\delta_i(1-\delta_i)(p_{i}-1)^{2}$ over $i$ gives the desired result.
\end{proof}

We now turn to the second part of the tail end. 

\begin{lem}\label{tailbound}
With the notation above, we have
\[
\sum_{i>2.5\cdot 10^8} \dfrac{M_i^{(2)}}{4\delta_i(1-\delta_i)}
< 0.000016385.
\]
\end{lem}

\begin{proof}
As in the previous proof, since $\delta_i=1/2$ for $i > 10^5$, we see that $4\delta_i(1-\delta_i)=1$ for each such expression in the sum. Our starting point is again the bound \eqref{lem2bd}
from Lemma~\ref{mainmomentbounds}.
We write the product there as
\begin{equation}\label{productsplit}
\prod_{j <i}\bigg(1+\frac{3p_j-1}{(1-\delta_j)(p_j-1)^2}\bigg)
= C_{0}\,
\prod_{2.5\cdot 10^8 <j<i}\bigg(1+\frac{3p_j-1}{(1-\delta_j)(p_j-1)^2}\bigg),
\end{equation}
where
\[C_0=\prod_{j \leq 2.5\cdot 10^8}\bigg(1+\frac{3p_j-1}{(1-\delta_j)(p_j-1)^2}\bigg).\]
A direct computation gives that $C_0 < 1048993.2048528$. 

We now look at the product on the right-hand side of \eqref{productsplit}. 
The function $x(3x-1)/(x-1)^2$ is decreasing for $x > 1$.  
For $j>2.5\cdot 10^8$, making use of $p_{2.5\cdot 10^8} = 5336500537$,
we obtain that 
\[\frac{p_j(3p_j-1)}{(1-1/2)(p_j-1)^2}<\frac{p_{2.5\cdot 10^8}(3p_{2.5\cdot 10^8}-1)}{(1-1/2)(p_{2.5\cdot 10^8}-1)^2} < 6.000000002,\]
and so 
\[\frac{3p_j-1}{(1-1/2)(p_j-1)^2}<\frac{6.000000002}{p_j}.\]
This implies that 
\[\prod_{2.5\cdot 10^8 <j<i}\bigg(1+\frac{3p_j-1}{(1-\delta_j)(p_j-1)^2}\bigg)<\prod_{2.5\cdot 10^8 <j<i}\bigg(1+\frac{6.000000002}{p_j}\bigg), \]
which we now proceed to estimate. We start by noting that 
\begin{align*}
\log\bigg(\prod_{2.5\cdot 10^8 <j<i}\bigg(1+\frac{6.000000002}{p_j}\bigg)\bigg)
&=\sum_{2.5\cdot 10^8 <j<i}\log\bigg(1+\frac{6.000000002}{p_j}\bigg) \\[5pt]
&< \sum_{2.5\cdot 10^8<j<i}\frac{6.000000002}{p_j}.
\end{align*}
To compute the rightmost sum, we use that
\[
\sum_{2.5\cdot 10^8<j<i}\frac{6.000000002}{p_j}=\sum_{j<i}\frac{6.000000002}{p_j}
- C_{1},
\]
where
\[
C_{1} = \sum_{j \leq 2.5\cdot 10^8}\frac{6.000000002}{p_j} > 20.22277221,
\]
done by a direct computation.

Recall that we are interested in $i > 2.5\cdot 10^8$ and $p_{2.5\cdot 10^8} = 5336500537$.  
From \cite[Theorem~5]{Rosser}, with $B = 0.26149721284764\ldots$, we deduce that
\begin{align*}
\sum_{j<i}\frac{6.000000002}{p_j} 
&< 6.000000002\bigg(\log\log p_{i-1} +B+\frac{1}{2\log^2(p_{i-1})}\bigg) \\[5pt]
&< 6.000000002\bigg(\log\log p_{i} +B+\frac{1}{2\log^2(p_{2.5\cdot 10^8})}\bigg) \\[5pt]
&< 6.000000002\,\log\log p_{i} + 1.574963388.
\end{align*}
Hence, we see that
\begin{align*}
\prod_{j <i}\bigg(1+\frac{3p_j-1}{(1-\delta_j)(p_j-1)^2}\bigg)
&< C_0\,(\log p_{i})^{6.000000002}\exp\big(1.574963388-C_1\big) \\[5pt]
&< 0.008358575 \,(\log p_{i})^{6.000000002}.
\end{align*}
Let $C_{2} = 0.008358575$. 
Putting everything together, we find that
\[
\sum_{i>2.5\cdot 10^8} \dfrac{M_i^{(2)}}{4\delta_i(1-\delta_i)} 
= \sum_{i>2.5\cdot 10^8}M_i^{(2)} 
< C_2 \sum_{i >2.5\cdot 10^8}\frac{(\log p_{i})^{6.000000002}}{(p_i-1)^{2}}.
\]
One can check that 
$(\log x)^{6.000000002}/(x-1)^{1/3}$
is a decreasing function for $x \ge p_{2.5\cdot 10^8}$.
Hence, we obtain 
\[
\sum_{i>2.5\cdot 10^8} \dfrac{M_i^{(2)}}{4\delta_i(1-\delta_i)}  
< \dfrac{C_{2}\,(\log p_{2.5\cdot 10^8})^{6.000000002}}{(p_{2.5\cdot 10^8}-1)^{1/3}}
\sum_{i >2.5\cdot 10^8}\frac{1}{(p_{i}-1)^{5/3}}.
\]
From \cite[Theorem~3]{Rosser}, we see that $p_i-1>i\log i$ for $i>2.5\cdot 10^8$. Thus, we deduce that
\begin{align*}
\sum_{i>2.5\cdot 10^8} \dfrac{M_i^{(2)}}{4\delta_i(1-\delta_i)}  
&< \dfrac{C_{2}\,(\log p_{2.5\cdot 10^8})^{6.000000002}}{(p_{2.5\cdot 10^8}-1)^{1/3}}
\sum_{i >2.5\cdot 10^8}\frac{1}{(i\log i)^{5/3}} \\[5pt]
&< \dfrac{C_{2}\,(\log p_{2.5\cdot 10^8})^{6.000000002}}{(p_{2.5\cdot 10^8}-1)^{1/3}(\log(2.5\cdot 10^8))^{5/3}}
\sum_{i >2.5\cdot 10^8}\frac{1}{i^{5/3}}.
\end{align*}
Using an integral to bound the sum, we find that
\[
\sum_{i >2.5\cdot 10^8}\frac{1}{i^{5/3}}<\int_{2.5\cdot 10^8}^\infty \frac{1}{x^{5/3}}\,dx=\frac{3}{2\,(2.5\cdot 10^8)^{2/3}}.
\]
We deduce that
\[
\sum_{i>2.5\cdot 10^8}\frac{M_i^{(2)}}{4\delta_i(1-\delta_i)}
< \dfrac{3 C_{2}\,(\log p_{2.5\cdot 10^8})^{6.000000002}}{2 (p_{2.5\cdot 10^8}-1)^{1/3}(\log(2.5\cdot 10^8))^{5/3} (2.5\cdot 10^8)^{2/3}}
< 0.000016385,
\]
which is what we wanted to show.  
\end{proof}

Now that we have finished the tail end estimates, we combine our second moment bounds.  
From \eqref{t1bound}, \eqref{t2bound}, Lemma~\ref{firsttailbound} and Lemma~\ref{tailbound}, we obtain
\begin{equation}\label{secondmomentbound}
\sum_{i \geq 19} \frac{M_i^{(2)}}{4\delta_i(1-\delta_i)} < 
\begin{cases}
0.305961222 &\text{if } q = 11 \\
0.299856390 &\text{if } 13 \le q \le 61 \\
0.996212323 &\text{if } q > 61.
\end{cases}
\end{equation}

\section{Proof of the Main Result and Discussion}
\begin{proof}[Proof of Theorem \ref{main}]
From \eqref{firstmomentbound} and \eqref{secondmomentbound}, we obtain 
\[
\sum_{\substack{1 \le i \le J\\ \delta_{i} = 0}} M_i^{(1)} + \sum_{\substack{1 \le i \le J\\ \delta_{i} \ne 0}} \min\bigg(M_i^{(1)},\frac{M_i^{(2)}}{4\delta_i(1-\delta_i)}\bigg) <
\begin{cases}
0.996826186 &\text{if } q = 11 \\
0.884434437 &\text{if } 13 \le q \le 61 \\
0.996212324 &\text{if } q > 61.
\end{cases}
\]
Since these bounds are all less than $1$, we deduce Theorem~\ref{main} from Lemma~\ref{maindistortion}. 
\end{proof}

Some natural questions arise from our results.  The following is motivated by Theorem~\ref{no11}.
 
\begin{question}
What is the largest integer $n$ for which it is possible to construct a covering system using only distinct multiples of $n$ as moduli?
\end{question}

Theorem~\ref{no11} shows that such an $n$ cannot be a prime $\ge 11$.  On the other hand, the work of Hough \cite{hough} and of Balister, Bollob\'as, Morris, Sahasrabudhe, and Tiba \cite{Balister} imply that such an $n$ is bounded.   A construction similar to that used in the proof for Corollary~\ref{cor} together with the work by Harrington in \cite{Harrington} implies that for $n = 4$, it is possible to construct a covering system using only distinct multiples of $n$ as moduli.  

Regarding Corollary~\ref{cor}, recall that we defined $N$ to be the largest integer for which it is possible to construct a family of $N$ distinct covering systems such that no two covering systems share a common modulus. We established that $N \le 9$, but the exact value of $N$ is still not clear. 

\begin{question}
What is the maximal $N$ such that a family of $N$ distinct covering systems exist where the sets of moduli in the covering systems are pairwise disjoint?
\end{question}

The above question is motivated by Harrington \cite{Harrington}.  
Short of finding the exact value of $N$, any improvement upon either the current best known lower bound $N \geq 3$, or the current best known upper bound $N \le 9$, would be of interest.  We note that the related problem of finding the maximal $N$ such that there is a system of congruences with distinct moduli $> 1$ for which every integer satisfies at least $N$ congruences in the system was posed in \cite{Harrington}.  It would be interesting to know whether these two values of $N$, the one in the question above and the one from \cite{Harrington}, are necessarily the same.

\subsection*{Funding} The third author is supported by a scholarship from the Natural Sciences and Engineering Research Council of Canada (NSERC). 

\subsection*{Acknowledgment}
The authors are grateful to Aditya Kumar for pointing out an error in our initial arguments, requiring us to adjust Lemma~\ref{T2bound}.  

\bibliographystyle{plain}

\begin{thebibliography}{99}

\bibitem{BalisterExposition}
P.~Balister, B.~Bollob\'as, R.~Morris, J.~Sahasrabudhe, and M.~Tiba, \textit{Erd\H{o}s covering systems}, Acta Math.~Hungar.~161 (2020), 540--549.

\bibitem{BalisterSquareFree}
P.~Balister, B.~Bollob\'as, R.~Morris, J.~Sahasrabudhe, and M.~Tiba, \textit{The Erd\H{o}s-Selfridge problem with square-free moduli}, Algebra and Number Theory 15 (2021), 609-–626.

\bibitem{Balister}
P.~Balister, B.~Bollob\'as, R.~Morris, J.~Sahasrabudhe, and M.~Tiba, \textit{On the Erd\H{o}s covering
problem: the density of the uncovered set}, Invent.~Math.~228 (2022), 377--414.

\bibitem{Cummings}
M.~Cumming, M.~Filaseta, and O.~Trifonov, \textit{An upper bound for the minimum modulus in a covering system with squarefree moduli}, Acta Math.~Hungar.~175 (2025), 1--25.

\bibitem{Erdos}
P.~Erd\H{o}s, \textit{On integers of the form $2^k+p$ and some related problems}, Summa Brasil.~Math.~2 (1950), 113--123.

\bibitem{Mirsky}
P.~Erdős,
 \textit{On a problem concerning congruence systems} (Hungarian),
Mat.~Lapok~3 (1952), 122--128.

\bibitem{Erdos2}
P.~Erd\H{o}s, \textit{Problems and results in combinatorial number theory, A survey of combinatorial theory}, J.~N.~Srivastava, ed., North-Holland, Amsterdam, 1973, 117--138.

\bibitem{Filaseta}
M.~Filaseta and A.~Kalogirou,
\textit{Covering systems with the sum of the reciprocals of the moduli close to $1$}, Trans.~Amer.~Math.~Soc.~379 (2026), 4945--4972.

\bibitem{Harrington}
J.~Harrington, \textit{Two questions concerning covering systems}, International Journal of Number Theory 11 (2015), 1739--1750.

\bibitem{hough}
R.~D.~Hough, \textit{Solution of the minimum modulus problem for covering systems}, Ann. of Math.  181 (2015), 361--382.

\bibitem{Klein}
J.~Klein, D.~Koukoulopoulos, and S.~Lemieux, \textit{On the $j$-th smallest modulus of a covering system with distinct moduli}, International Journal of Number Theory 20 (2024), 471--479.

\bibitem{Li}
H.~Li, B.~Wang, C.~Wang, and S.~Yi, \textit{On covering systems of polynomial rings over finite fields}, Proc.~Amer.~Math.~Soc.~152 (2024), 3731--3742.

\bibitem{Li2}
 H.~Li, B.~Wang, C.~Wang, and S.~Yi, \textit{On Erd\H{o}s covering systems in global function fields}, J.~Number Theory 266 (2025), 269--280.

 \bibitem{Li3}
 H.~Li, B.~Wang, and S.~Yi, \textit{On the minimum modulus problem in number fields}, arXiv preprint, arXiv:2302.05946, 2026.

\bibitem{Rosser}
J.~B.~Rosser and L.~Schoenfeld, \textit{Approximate formulas for some functions of prime numbers},
Illinois J.~Math.~6 (1962), 64--94.

\bibitem{Wang}
B.~Wang, \textit{Bounds for Erd\H{o}s covering systems in global function fields}, International Journal of Number Theory 22 (2026), 1327--1336. 

\end{thebibliography}

\end{document}